\documentclass[10pt, english]{article}

\usepackage[margin= 2 cm]{geometry}

\usepackage{amsthm}
\usepackage{amsmath}
\usepackage{amssymb}
\usepackage{setspace}
\usepackage{mathtools}
\usepackage{verbatim}
\usepackage{csquotes}
\usepackage{graphicx}
\usepackage[hidelinks]{hyperref}
\usepackage{thm-restate}
\usepackage{cleveref}
\usepackage{graphicx}
\usepackage{appendix}
\usepackage[inline]{enumitem}
\usepackage{framed}
\usepackage{subcaption}

\usepackage{caption}

\usepackage{floatrow}
\usepackage[T1]{fontenc}

\usepackage{tikz}
\usepackage{mathdots}
\usepackage{xcolor}
\usepackage{diagbox}
\usepackage{colortbl}
\usepackage[absolute,overlay]{textpos}

\graphicspath{ {./images/} }

\usetikzlibrary{calc}
\usetikzlibrary{decorations.pathreplacing}
\usetikzlibrary{positioning,patterns}
\usetikzlibrary{arrows,shapes,positioning}
\usetikzlibrary{decorations.markings}

\tikzstyle{edge}=[very thick]
\definecolor{bostonuniversityred}{rgb}{0.8, 0.0, 0.0}
\definecolor{arsenic}{rgb}{0.23, 0.27, 0.29}
\tikzstyle{diredge}=[postaction={decorate,decoration={markings,
		mark=at position .95 with {\arrow[scale = 1]{stealth};}}}]

\newcommand{\defPt}[3]{
	\def \pt {(#1, #2)}
	\coordinate [at = \pt, name = #3];
}

\tikzset{
   conn/.pic={
     \defPt{0.2}{-0.5}{q0}
     \defPt{-1}{-1.5}{q5}
    \defPt{1}{1.2}{q1}
    \defPt{1}{2.7}{q6}
    \defPt{1.25}{-1.2}{q2}
    \defPt{2.5}{0.6}{q3}
    \defPt{2.5}{-0.6}{q4}
  
        \draw[line width=1 pt] (q0) -- (q1) -- (q3) -- (q4);
        \draw[line width=1 pt] (q2) -- (q3);
        \draw[line width=1 pt] (q0) -- (q5);
        \draw[line width=1 pt] (q1) -- (q6);
  }
}

\newcommand{\fitellipsis}[3] 
{\draw []let \p1=(#1), \p2=(#2), \n1={atan2(\y2-\y1,\x2-\x1)}, \n2={veclen(\y2-\y1,\x2-\x1)}
    in ($ (\p1)!0.5!(\p2) $) ellipse [ x radius=\n2/2+0.3cm+#3cm, y radius=#3cm, rotate=\n1];
}

\newcommand{\fitellipsiss}[3] 
{\draw [fill=white]let \p1=(#1), \p2=(#2), \n1={atan2(\y2-\y1,\x2-\x1)}, \n2={veclen(\y2-\y1,\x2-\x1)}
    in ($ (\p1)!0.5!(\p2) $) ellipse [ x radius=\n2/2+#3cm, y radius=#3cm, rotate=\n1];
}

\newcommand{\fitellipsisss}[3] 
{\draw []let \p1=(#1), \p2=(#2), \n1={atan2(\y2-\y1,\x2-\x1)}, \n2={veclen(\y2-\y1,\x2-\x1)}
    in ($ (\p1)!0.5!(\p2) $) ellipse [ x radius=\n2/2+#3cm, y radius=#3cm, rotate=\n1];
}

\theoremstyle{plain}

\newtheorem*{thm*}{Theorem}
\newtheorem{thm}{Theorem}
\Crefname{thm}{Theorem}{Theorems}

\newtheorem*{lem*}{Lemma}
\newtheorem{lem}[thm]{Lemma}
\Crefname{lem}{Lemma}{Lemmas}

\newtheorem*{claim*}{Claim}

\Crefname{claim}{Claim}{Claims}
\Crefname{claim}{Claim}{Claims}

\newtheorem{prop}[thm]{Proposition}
\Crefname{prop}{Proposition}{Propositions}

\newtheorem{cor}[thm]{Corollary}
\Crefname{cor}{Corollary}{Corollaries}

\newtheorem{conj}[thm]{Conjecture}
\Crefname{conj}{Conjecture}{Conjectures}

\Crefname{qn}{Question}{Questions}

\Crefname{obs}{Observation}{Observations}

\newtheorem{ex}[thm]{Example}
\Crefname{ex}{Example}{Examples}

\theoremstyle{definition}

\Crefname{prob}{Problem}{Problems}

\Crefname{defn}{Definition}{Definitions}

\newtheorem*{defn*}{Definition}

\theoremstyle{remark}

\newcommand{\ceil}[1]{
    \left\lceil #1 \right\rceil
}

\newcommand{\eps}{\varepsilon}

\expandafter\def\expandafter\normalsize\expandafter{%
    \normalsize
    \setlength\abovedisplayskip{8pt}
    \setlength\belowdisplayskip{8pt}
    \setlength\abovedisplayshortskip{4pt}
    \setlength\belowdisplayshortskip{4pt}
}

\usepackage[square,sort,comma,numbers]{natbib}
 \setlist[itemize]{leftmargin=*}

\DeclareFontFamily{OT1}{pzc}{}
\DeclareFontShape{OT1}{pzc}{m}{it}{<-> s * [1.10] pzcmi7t}{}
\DeclareMathAlphabet{\mathpzc}{OT1}{pzc}{m}{it}

\title{\vspace{-0.8cm}  Equivalence between Erd\H{o}s-Hajnal and polynomial R\"odl and Nikiforov conjectures}

\author{ 
Matija Buci\'c\thanks{School of Mathematics, Institute for Advanced Study and Department of Mathematics, Princeton University, Princeton, USA. Email: \href{mailto:matija.bucic@ias.edu} {\nolinkurl{matija.bucic@ias.edu}}.} \and Jacob Fox\thanks{Department of Mathematics, Stanford University, Stanford, CA. Email: \textbf{jacobfox@stanford.edu}. Research
supported by NSF Awards DMS-2154129.} 
\and Huy Tuan Pham\thanks{Department of Mathematics, Stanford University, Stanford, CA. Email: \textbf{huypham@stanford.edu}. Research supported by a Clay Research Fellowship and a Stanford Science Fellowship.}}
 \date{}

\begin{document}

\maketitle

\begin{abstract}
It is well-known that polynomial versions of theorems of R\"odl and Nikiforov, as conjectured by Fox and Sudakov and Nguyen, Scott and Seymour imply the classical Erd\H{o}s-Hajnal conjecture. In this note, we prove that these three conjectures are in fact equivalent, extending several previous particular results in this direction by Fox, Nguyen, Scott and Seymour; Nguyen, Scott and Seymour and Gishboliner and Shapira. We deduce that the family of string graphs satisfies the polynomial R\"odl conjecture. We also derive analogous results for hypergraphs, tournaments, ordered graphs, and colored graphs. 
\end{abstract}

\section{Introduction} 

A subset of vertices of a graph is {\it homogeneous} if it is either a clique or an independent set. For a graph $G$, denote by $\hom(G)$ the size of the largest homogeneous subset of vertices of $G$. A graph is \textit{$H$-free} if it does not contain $H$ as an induced subgraph. 

Among the most prominent open problems in extremal and structural combinatorics is the following conjecture of  Erd\H{o}s and Hajnal \cite{erdos-hajnal-1,erdos-hajnal-2} from 1977.
\begin{conj}\label{conj:E-H}
For every graph $H$ there is $c_H>0$ such that any $n$-vertex $H$-free graph $G$ has $\hom(G) \ge n^{c_H}$.
\end{conj}

The Erd\H{o}s-Hajnal conjecture implies that $H$-free graphs contain much larger homogeneous sets than typical graphs. Indeed, Erd\H{o}s \cite{Erdos1947} in 1947 proved that a random graph $G$ on $n$ vertices asymptotically almost surely satisfies $\hom(G) \leq 2\log_2 n$. In the direction of their conjecture, Erd\H{o}s and Hajnal \cite{erdos-hajnal-2} proved that every $H$-free graph $G$ on $n$ vertices satisfies $\hom(G) \geq e^{c_H \sqrt{\log n}}$. This bound was recently improved by Bucic, Nguyen, Scott and Seymour \cite{tung-paper} to $\hom(G) \geq e^{c_H \sqrt{\log n \log \log n}}$.

A graph $H$ satisfying \Cref{conj:E-H} is said to have the \emph{Erd\H{o}s-Hajnal property}. Graphs with up to five vertices have the Erd\H{o}s-Hajnal property (combining results from \cite{erdos-hajnal-2,APS01,gyarfas-survey,bull,5-hole,P_5-full}). Alon, Pach, and Solymosi \cite{APS01} proved that the Erd\H{o}s-Hajnal property holds for any graph that can be obtained by substitution from smaller graphs which have the Erd\H{o}s-Hajnal property. Nguyen, Scott, and Seymour \cite{buildable} recently proved that an infinite family of prime graphs (not obtainable by substitution from smaller graphs) each satisfy the Erd\H{o}s-Hajnal property. An approximate version of the Erd\H{o}s-Hajnal conjecture was recently established for paths \cite{long-paths}. For more on the Erd\H{o}s-Hajnal conjecture, see the surveys \cite{gyarfas-survey,fox-pach-survey,maria-survey}. Various variants of the Erd\H{o}s-Hajnal property and their relationships to each other have been studied before, connecting Ramsey-type properties to Tur\'an-type properties and Szemer\'edi regularity-type properties (see \cite{fps}). This paper answers some questions about the relationships between some of these properties.

Fox and Sudakov \cite{fox-sudakov} in 2008 conjectured the following quantitative strengthening of a classical result of R\"odl \cite{rodl} from 1986. A set $S$ of vertices in a graph $G$ is 
\emph{$\eps$-homogeneous} if the edge density of the induced subgraph $G[S]$ is at most $\eps$ or at least $1-\eps$. 

\begin{conj}\label{conj:rodl}
For any graph $H$ there exists a $C_H>0$ such that for any $0<\eps \le 1/2$ any $n$-vertex $H$-free graph contains an $\eps$-homogeneous set with at least $\eps^{C_H}n$ vertices.
\end{conj}

A graph $H$ satisfying \Cref{conj:rodl} is said to have the \emph{polynomial R\"odl property}. It is not difficult to show that a graph satisfying the polynomial R\"odl property also has the Erd\H{o}s-Hajnal property through an application of Tur\'an's theorem. Indeed, suppose \Cref{conj:rodl} holds for $H$ with constant $C_H$ and let $c_H=1/(C_H+1)$. Suppose $G$ is an $H$-free graph on $n$ vertices, and let $\eps=n^{-c_H}$. By \Cref{conj:rodl}, there is an $\eps$-homogeneous set $S$ with  
$|S| \geq \eps^{C_H}n=n^{c_H}$. Without loss of generality, the edge density of $G[S]$ is at least $1-\eps$. Applying Tur\'an's theorem, there is a clique in $G[S]$ (and hence in $G$) of order $n^{c_H}/2$. Thus if a graph $H$ has the polynomial R\"odl property with associated constant $C_H$, then it also has the Erd\H{o}s-Hajnal property with constant $c_H=1/(C_H+1)$ (apart from a factor two in the homogeneous set size). 

Nikiforov \cite{Nikiforov} proved a strengthening of R\"odl's theorem that allows for the same conclusion under the weaker hypothesis of not containing too many induced copies of $H$. Fox, Nguyen, Scott and Seymour \cite{P_4-rodl} conjectured the following quantitative strengthening of Nikiforov's theorem and of the polynomial R\"odl conjecture.

\begin{conj}\label{conj:nikiforov}
For any graph $H$ there are constants $C_H,D_H$ such that for any $0<\eps \le 1/2$, any $n$-vertex graph with at most $\eps^{D_H}n^{|H|}$ induced copies of $H$ contains an $\eps$-homogeneous set with at least $\eps^{C_H}n$ vertices.
\end{conj}

We refer to \Cref{conj:nikiforov} as the polynomial Nikiforov conjecture. A graph $H$ is \emph{viral} if it satisfies \Cref{conj:nikiforov}. Every graph which is known to have the Erd\H{o}s-Hajnal property is also known to be viral (see \cite{P_4-rodl,rodl-C_5,buildable,P_5-full}).  Some of these proofs show only that these particular graphs have the polynomial R\"odl property, but they can also be tweaked to show these graphs are viral. Furthermore, the previously mentioned approximate results for the Erd\H{o}s-Hajnal conjecture extend to approximate versions of the polynomial Nikiforov conjecture. It is proved in \cite{P_4-rodl} that being viral is closed under substitution. 

Nguyen, Scott, and Seymour mentioned in several papers \cite{long-paths,buildable} that it is an open problem whether the polynomial R\"odl property is closed under substitution. A positive solution to this problem follows as a corollary of Theorem \ref{mainthm1} below. The relationship between the Erd\H{o}s-Hajnal property, the polynomial R\"odl property, and being viral was recently discussed in detail in \cite{rodl-C_5}, and has been mentioned in many papers over the years \cite{fps,fox-sudakov,tung-paper,P_4-rodl,P_5-full,long-paths,VC-dimension,lior-P4,buildable}. We prove that these three properties are in fact equivalent.

\begin{thm}\label{mainthm1} If a graph has the Erd\H{o}s-Hajnal property, then it is also viral. Hence, a graph has the Erd\H{o}s-Hajnal property if and only if it has the polynomial R\"odl property if and only if it is viral. 
\end{thm}

This extends the results of \cite{P_4-rodl, rodl-C_5,lior-P4} which were establishing that this holds for $P_4$ in \cite{P_4-rodl}, giving a simpler proof of this together with an extension of this result to certain special families of graphs in \cite{lior-P4} and establishing it for all graphs on up to five vertices except $P_5$ (and its complement) in \cite{rodl-C_5}. It also gives a very short proof that $P_5$ is viral (when combined with the recent breakthrough \cite{P_5-full} showing $P_5$ has the Erd\H{o}s-Hajnal property), a result which was announced in \cite{P_5-full} to appear in Nguyen's thesis in 2025.  

We note that while Theorem \ref{mainthm1} shows that the three properties are equivalent there is still a hierarchy of how much power they give one when used. Indeed, many of the recent breakthroughs on the Erd\H{o}s-Hajnal conjecture crucially exploit the additional power one gets from the stronger polynomial R\"odl property. Our result however, allows one to only prove the Erd\H{o}s-Hajnal property while still having the full power of the polynomial Nikiforov property assumed as part of an inductive or iterative argument. This allows for slight simplification of some of the arguments of this flavour behind the recent progress towards the Erd\H{o}s-Hajnal conjecture mentioned above and we hope could be useful in making further progress towards the conjecture.

Theorem \ref{mainthm1} follows from a more precise version that gives a sharp dependence on the constant between the Erd\H{o}s-Hajnal property and the polynomial R\"odl property. 

\begin{thm}\label{mainthmquant}
Let $H$ be a graph on $h$ vertices with the Erd\H{o}s-Hajnal property with constant $c_H$. Then for every $0<\eps<1$, every graph $G$ on $n$ vertices contains at least $\eps^{(h-1)/c_H+o(1)}n^h$ induced copies of $H$ or contains an $\eps$-homogeneous set of size $\eps^{1/c_H -1+o(1)}n$.
\end{thm}

The $\eps^{o(1)}$ factors in Theorem \ref{mainthmquant} are actually polylogarithmic in $\eps^{-1}$ factors. Also recalling the deduction of the Erd\H{o}s-Hajnal property from the polynomial R\"odl property, apart from the polylogarithmic factor, the dependence between the constants in the Erd\H{o}s-Hajnal property and the polynomial R\"odl property given by \Cref{mainthmquant} is tight, and only obtains a polylogarithmic factor loss in going back and forth between these properties. 

\Cref{mainthmquant} also gives the sharp constant in some cases for the number of allowed induced copies of $H$. If $H=P_4$, then Theorem \ref{mainthmquant} gives that every graph on $n$ vertices has at least $\eps^{6+o(1)}n^4$ induced copies of $P_4$ or contains an $\eps$-homogeneous set of size $\eps^{1+o(1)}n$. These bounds are sharp apart from the $o(1)$ terms by the following proposition. 

\begin{prop}\label{prop-P_4}
Let $0<\eps<1/2$. For every integer $n>1$ there is a graph on $n$ vertices with  $O(\eps^6 n^4)$ induced copies of $P_4$ and every $\eps$-homogeneous set has size at most $O_{\eps}(\log n)$. 
\end{prop}

The proof of \Cref{prop-P_4} goes by the following construction. If $\eps$ is not below some small positive constant $\eps_0$, we can let $G=G(n,1/2)$ which almost surely has the desired properties. Otherwise, by considering an Erd\H{o}s-Renyi random graph with density $2\eps$, there is a graph $G_0$ on $n$ vertices such that every induced subgraph with at least $C_{\eps}\log n$ vertices has edge density greater than $\eps$ and less than  $3\eps$. Equitably vertex partition $G_0$ into $s=1/(5\eps)$ parts and add all edges between distinct parts to obtain $G$. Every induced subgraph of $G$ on at least $C_{\eps}\log n$ vertices has edge density greater than $\eps$ since we only added edges to obtain $G$ from $G_0$. Note that the edge density of a complete $s$-partite graph is maximized if the partition is equitable, and it has edge density $1-1/s$ in that case. So the edge density of any induced subgraph on at least $C_{\eps}\log n$ vertices is at most $1-1/s+3\eps =1-2\eps$. The only induced copies of $P_4$ are inside the parts, and the density of $P_4$ in each such part is $O(\eps^3)$. It follows that the number of induced copies of $P_4$ in $G$ is at most $O(\eps^6 n^4)$. 

A family of graphs is \emph{hereditary} if it is closed under taking induced subgraphs. A family $\mathcal{F}$ of graphs has the \emph{Erd\H{o}s-Hajnal property} if there is $c_{\mathcal{F}}>0$ such that every graph $G \in \mathcal{F}$ on $n$ vertices has $\hom(G) \geq n^{c_{\mathcal{F}}}$. Thus saying a graph $H$ has the Erd\H{o}s-Hajnal property is equivalent to the family $\mathcal{F}_H$ of $H$-free graphs has the Erd\H{o}s-Hajnal property. The Erd\H{o}s-Hajnal conjecture is equivalent to every proper hereditary family of graphs having the Erd\H{o}s-Hajnal property. 

A family $\mathcal{F}$ of graphs has the \emph{polynomial R\"odl property} if there is $C_{\mathcal{F}}$ such that for every $0<\eps<1/2$, every graph $G \in \mathcal{F}$ on $n$ vertices has an $\eps$-homogeneous set $S$ with $|S| \geq \eps^{C_{\mathcal{F}}}n$. The proof that the Erd\H{o}s-Hajnal property for a graph $H$ is equivalent to the polynomial R\"odl property readily generalizes with the same essentially sharp quantitative dependencies.  

\begin{thm}\label{mainfamily}
A hereditary family of graphs has the 
Erd\H{o}s-Hajnal property if and only if it has the polynomial R\"odl property.
\end{thm}

A \emph{string graph} is an intersection graph of curves in the plane. Tomon \cite{tomon2023string} proved that the family of string graphs has the Erd\H{o}s-Hajnal property. We therefore obtain the following corollary to \Cref{mainfamily}. 

\begin{cor}\label{cor:string}
The family of string graphs has the polynomial R\"odl property. That is, there is a constant $C$ such that the following holds. For any finite family $F$ of curves in the plane and any $\varepsilon>0$, there is a subfamily $F'$ of at least $\varepsilon^C |F|$ curves such that either all but at most $\varepsilon |F'|^2$ pairs of curves in $F'$ intersect, or all but at most $\varepsilon |F'|^2$ pairs of curves in $F'$ are disjoint. 
    \end{cor} 

A graph is \emph{perfect} if all of its induced subgraphs have equal clique number and chromatic number. The celebrated strong perfect graph theorem of Chudnovsky, Robertson, Seymour, and Thomas \cite{CRST} characterizes the family of perfect graphs by forbidden induced subgraphs. The fact that the family of perfect graphs has the polynomial R\"odl property follows from the fact that it is a subfamily of $\mathcal{F}_{C_5}$, which was recently shown to have the polynomial R\"odl property  \cite{rodl-C_5}. Every perfect graph $G$ on $n$ vertices has $\hom(G) \geq \sqrt{n}$, and the following is an immediate  corollary of the quantitative version of Theorem \ref{mainfamily}. With a short proof, it gives a sharp up to the polylogarithmic factor bound for the polynomial R\"odl property for perfect graphs. 

\begin{cor}\label{corperfect}
For every $\eps>0$ every perfect graph on $n$ vertices has an $\eps$-homogeneous set of size $\Omega(\eps(\log 1/\eps)^{-2}n)$. 
    \end{cor}

A graph $G$ on $n$ vertices is \emph{$\alpha$-close} to a family $\mathcal{F}$ of graphs if one can add or delete at most $\alpha n^2/2$ edges from $G$ to obtain a graph in $\mathcal{F}$. A natural conjecture for hereditary families is the following. 

\begin{conj}\label{closeconj}
For each proper hereditary family $\mathcal{F}$ of graphs there are $C_{\mathcal{F}},d_{\mathcal{F}}>0$ such that if $0<\eps<1/2$, every graph on $n$ vertices which is $\eps^{d_{\mathcal{F}}}$-close to $\mathcal{F}$ contains an $\eps$-homogeneous set of size at least $\eps^{C_{\mathcal{F}}}n$. 
\end{conj}

We say a hereditary family $\mathcal{F}$ of graphs which satisfies \Cref{closeconj} has the \emph{polynomial close R\"odl property}. Clearly, a family with the polynomial close R\"odl property also has the polynomial R\"odl property. The proof of \Cref{{mainfamily}} extends to the following theorem. 

\begin{thm}\label{mainfamily2}
A hereditary family of graphs has the 
Erd\H{o}s-Hajnal property if and only if it has the polynomial close R\"odl property.
\end{thm}

Several variants of the Erd\H{o}s-Hajnal conjecture and its aforementioned strengthenings have been studied for other combinatorial structures, including for hypergraphs, tournaments, ordered graphs, and edge-colored complete graphs. In each of these settings, one can define the natural analogs of the Erd\H{o}s-Hajnal property, the polynomial R\"odl property and the viral property for the corresponding combinatorial structure. Our results extend to show the equivalence of these properties for the aforementioned combinatorial structures, as well as an analogous result for hypergraphs. We discuss these in detail in \Cref{extensions}.  

\textbf{Organization and Notation.} The next section contains the proofs of our main theorems. In \Cref{extensions}, we discuss natural extensions of our results to hypergraphs, ordered graphs, tournaments, and edge-colorings. We finish with some concluding remarks. All our logarithms are in base $e$ unless otherwise specified.

\section{Proofs}

The Kleitman-Winston graph container method \cite{KW80,KW82} and its hypergraph generalization \cite{BMS,ST} is a powerful tool in extremal combinatorics. It gives an upper bound on the number of independent sets in a graph of a given size. It has found many significant applications such as the recent breakthrough lower bound on off-diagonal Ramsey numbers \cite{MV}. The next lemma is a slight modification of the standard lemma, see e.g.\ \cite[Lemma 3.1]{KLRS} and \cite[Lemma 1]{counting-independent-sets}. We present a short alternative proof by induction.

\begin{lem}\label{lem:counts-upr-bnd}
Let $0 \leq \ell \leq k$ be integers, $u,n$ be positive integers, and $0< \varepsilon \leq 1$ be such that $(1-\varepsilon)^{\ell}n \leq u$. Suppose $G$ is a graph on $n$ vertices in which any vertex subset $S$ with $|S| \geq u$ is such that the induced subgraph $G[S]$ has maximum degree at least $\varepsilon |S|-1$. Then the number of independent sets of order $k$ in $G$ is at most $$\binom{n}{\ell}\binom{u}{k-\ell}.$$
\end{lem}
\begin{proof}
The proof is by induction on $n$. The base case is $n \leq u$. In the base case, there are trivially at most $\binom{n}{k}$ independent sets of order $k$, which is at most the desired bound. 

Suppose $G$ has $n>u$ vertices, so $\ell \geq 1$. Let $v$ be a vertex of $G$ of maximum degree, so $v$ has a least $\varepsilon n-1$ neighbors. Every independent set in $G$ of order $k$ either contains $v$ or does not contain $v$. We bound the number of independent sets of each type in turn. 

Let $U$ be the set of vertices of $G$ that are not $v$ and not adjacent to $v$. Then $|U| \leq (1-\varepsilon)n$. The independent sets in $G$ containing $v$ of order $k$ are, after deleting $v$, precisely the independent sets in $G[U]$ of order $k-1$. We can apply the induction hypothesis to $G[U]$, which has $|U|$ vertices and parameters $\ell-1$ and $k-1$. Note that we can apply the induction hypothesis as $(1-\varepsilon)^{\ell-1}|U| \leq (1-\eps)^{\ell}n \leq u$. Thus the number of independent sets in $G[U]$ of order $k-1$, which is also the number of independent sets in $G$ of order $k$ containing $v$, is at most 
\begin{equation*}
\binom{|U|}{\ell-1}\binom{u}{(k-1)-(\ell-1)} \leq \binom{n-1}{\ell-1}\binom{u}{k-\ell}.
\end{equation*}

Again using the induction hypothesis (with $n$ replaced by $n-1$), the number of independent sets in $G$ not containing $v$ is at most 
 \begin{equation*}\binom{n-1}{\ell}\binom{u}{k-\ell}.\end{equation*}
Adding together the above bounds and using Pascal's identity gives the desired result. 
\end{proof}

Note that the proof of Lemma \ref{lem:counts-upr-bnd} was wasteful in that we bounded $\binom{|U|}{\ell-1} \leq \binom{n-1} {\ell-1}$, whereas we know $|U| \leq (1-\eps)n$. One can easily adapt the inductive proof to obtain a better upper bound on the number of independent sets of order $k$. One such bound is $2^{\ell}(u/n)^{(\ell-1)/2}\binom{n}{\ell}\binom{u}{k-\ell}$.

A graph has the \emph{$(t,k)$-homogeneous property} if every subset of $t$ vertices contains a homogeneous set of order $k$. 
\begin{lem}\label{generalhom}
Let $\mathcal{F}$ be a hereditary family of graphs such that every graph in $\mathcal{F}$ has the $(t,k)$-homogeneous property. Let $G$ be a graph on $n$ vertices such that, with probability at least $1/2$, a random subset of $2t$ vertices of $G$ contains an induced subgraph in $\mathcal{F}$ on $t$ vertices. Then $G$ contains at least $(1/2)(n/(2t))^k$ homogeneous sets of order $k$. 
\end{lem}
\begin{proof}
From the assumption, at least $(1/2) \binom{n}{2t}$ subsets of vertices of size $2t$ contain a homogeneous set of order $k$. Each homogeneous set of order $k$ is in at most $\binom{n-k}{2t-k}$ subsets of order $2t$. So the number of homogeneous sets of order $k$ is at least 
$$(1/2)\frac{\binom{n}{2t}}{\binom{n-k}{2t-k}}=(1/2)\frac{\binom{n}{k}}{\binom{2t}{k}} \geq (1/2)(n/2t)^k.$$
\end{proof}
The following lemma is proved by a sampling argument. 

\begin{lem} \label{lem:counts-lwr-bnd}
Let $H$ be a graph on $h$ vertices. Suppose every $H$-free graph has the $(t,k)$-homogeneous property. If a graph $G$ on $n\ge 2t$ vertices has at most $n^h/(2^{h+1}t^{h-1})$ induced copies of $H$, then $G$ contains at least $\frac{1}{2}(n/(2t))^k$ homogeneous sets of order $k$. 
\end{lem}

\begin{proof}
Consider a uniformly random subset $S$ of $2t$ vertices from $G$. The expected number of induced copies of $H$ in $G[S]$ is at most $(n^h/(2^{h+1} t^{h-1})) \binom{2t} {h}/\binom{n}{h} \leq \frac{t}{2}$. Hence, by Markov's inequality, with probability at least $\frac12$, the induced subgraph $G[S]$ has at most $t$ induced copies of $H$. Deleting one vertex from each of these induced copies, we get an induced subgraph with $t$ vertices which does not contain $H$ as an induced subgraph. This shows that $G$ satisfies the assumption of Lemma \ref{generalhom}. 

By Lemma \ref{generalhom}, we obtain that $G$ contains at least $\frac{1}{2}(n/(2t))^k$ homogeneous sets of order $k$.
\end{proof}

The following variant has a very similar proof. 

\begin{lem} \label{lem:counts-lwr-bnd-close}
Let $\mathcal{F}$ be a hereditary family of graphs such that every graph in $\mathcal{F}$ has the $(t,k)$-homogeneous property. If a graph $G$ on $n\ge 2t$ vertices is $t^{-1}$-close to $\mathcal{F}$, then $G$ contains at least $\frac{1}{2}(n/(2t))^k$ homogeneous sets of order $k$. 
\end{lem}
\begin{proof}
Since $G$ is $t^{-1}$-close to $\mathcal{F}$, there is a set $P$ of at most $t^{-1} n^2/2$ pairs of vertices such that adding or deleting the pairs in $P$ results in a graph in $\mathcal{F}$. Consider a uniformly random subset $S$ of $2t$ vertices from $G$. The expected number of pairs in $P$ with both vertices in $S$ is $|P|\binom{t}{2}/\binom{n}{2} \leq t/2$. 
Hence, by Markov's inequality, with probability at least $\frac12$, the induced subgraph $G[S]$ has at most $t$ pairs in $P$. Deleting one vertex from each of these pairs, we get an induced subgraph with $t$ vertices which is in $\mathcal{F}$. This shows that $G$ satisfies the assumption of Lemma \ref{generalhom}. 

By Lemma \ref{generalhom}, $G$ contains at least $\frac{1}{2}(n/(2t))^k$ homogeneous sets of order $k$. 
\end{proof}

Given a non-decreasing function $f:\mathbb{R}^+ \to \mathbb{R}^+$ such that $2 \le f(k) \le \log k$ for all $k \geq 1$ we say that a graph $H$ has the $f$-Erd\H{o}s-Hajnal property if for any $k \geq 1$ and any $H$-free graph $G$ with at least $k^{f(k)}$ vertices has a homogeneous set of size at least $k$. In particular, a graph $H$ has the Erd\H{o}s-Hajnal property if it has the $f$-Erd\H{o}s-Hajnal property with $f$ being a constant function. We make the assumption $f(k) \le \log k$ since the recent result from \cite{tung-paper} tells us that any graph $H$ satisfies the $O_H(\log k /\log \log k)$-Erd\H{o}s-Hajnal property.

The following theorem shows that if a graph $H$ has the Erd\H{o}s-Hajnal property, then it also has the viral property (and hence also the polynomial R\"odl property). Furthermore, the quantitative bounds from the theorem are essentially sharp (up to lower order factors) in going between the Erd\H{o}s-Hajnal and R\"odl properties. 

It is actually somewhat more general since it allows us to conclude weaker variants of said conjectures if we only have a weaker result towards the Erd\H{o}s-Hajnal conjecture. In particular, it can be used to recover a result of Fox and Sudakov \cite{fox-sudakov} showing a Nikiforov analog of the original bound of Erd\H{o}s and Hajnal. It can also be used to give a different proof of the Nikiforov analog of the current state of the art general bound from \cite{tung-paper} on the Erd\H{o}s-Hajnal conjecture as well as a different way of concluding that any path is near viral (see \cite{long-paths} for a definition and more details) starting from the near Erd\H{o}s-Hajnal property established in \cite{long-paths}.

\begin{thm}\label{thm:rodlfromerdos}
    Let $0<\eps< \frac{1}{100}$ and $k:=200\ceil{\frac1{\eps}\log^4 \frac1{\eps}}$. Let $1/\delta=16k^{f(k)-1}$. 
    Let $H$ be an $h$-vertex graph which has the $f$-Erd\H{o}s-Hajnal property. Let $G$ be an $n$-vertex graph with less than $(\delta /k)^{h-1} n^{h}$ induced copies of $H$. Then $G$ has an induced subgraph $G[S]$ with $|S| \ge \delta n$ with maximum degree at most $\eps (|S|-1)$ or minimum degree at least $(1-\eps)(|S|-1)$. 
\end{thm}

The proof uses \Cref{lem:counts-lwr-bnd} to get a lower bound on the number of homogeneous sets of order $k$, at least half of which we may assume are independent sets. Assuming the conclusion of the theorem is false, the proof uses Lemma \ref{lem:counts-upr-bnd} to get an upper bound on the number of independent sets of order $k$. The proof concludes by observing these lower and upper bounds on the number of independent sets of order $k$ contradict each other.

\begin{proof}[Proof of \Cref{thm:rodlfromerdos}] We begin by establishing several useful inequalities between some of the parameters we need.

We may assume $\delta n \ge 1$ as otherwise, the desired $S$ can be any single vertex subgraph. 
We may assume $n < (\delta n)^{f(\delta n)}$ as otherwise $G$ contains a homogeneous set of order $\delta n$, and we are done. If $\delta n < k$, as $f$ is non-decreasing, we have $(\delta n)^{f(k)} \geq (\delta n)^{f(\delta n)} > n$ and hence $n \geq \delta^{-1-1/(f(k)-1)} >  k/\delta$, a contradiction. So $n \geq k/\delta = 16k^{f(k)}$. 

Let $t = \ceil{k^{f(k)}},$ so that $n \ge 2t$ and $\frac1{\delta} > \frac{15t}k.$ Furthermore, by our assumption on $H$ being $f$-Erd\H{o}s-Hajnal every $H$-free graph has the $(t,k)$-homogeneous property.
Since $(\delta/k)^{h-1} < \left(15t\right)^{-(h-1)}\le 1/(2^{h+1}t^{h-1}),$ and $n \ge 2t$, we may apply \Cref{lem:counts-lwr-bnd} to obtain that $G$ has at least $\frac12 (n/t)^k$ homogeneous sets of order $k$. Without loss of generality, at least half of these homogeneous sets of order $k$ are independent sets.

We may assume that there is no induced subgraph $G[S]$ with $|S| \ge \delta n$ and maximum degree at most $\eps |S|-1$ as otherwise we would be done.  Let $\ell:=\ceil{\frac1{\eps}\log \frac1{\delta}}$ so that $(1-\eps)^\ell \le e^{-\eps\ell} \le \delta .$ By \Cref{lem:counts-upr-bnd} there are at most $\binom{n}{\ell}\binom{\delta n}{k-\ell}$ independent sets of order $k$ in $G$. 

Putting the two bounds on the number of independent sets of order $k$ together we get
\begin{align}
\frac{n^k}{4t^k} \le \binom{n}{\ell}\binom{\delta n}{k-\ell} \le \frac{\delta^{k-\ell}n^{k}}{\ell! (k-\ell)!}\le \frac{2^k\delta^{k-\ell}n^{k}}{k!}\le \frac{(2e)^k\delta^{k-\ell}n^{k}}{4k^k} \implies (1/\delta)^{1-\ell/k} \le \frac{2et}{k} \implies \frac{1}{\delta}\le \frac{2e^2t}k, \nonumber
\end{align}
with the last implication coming from $k/\ell \geq \log(1/\delta)$ which follows from substituting in the values of $k,\ell$ and $1/\delta$ and using $f(k) \leq \log k$. 
However, this contradicts $\frac1{\delta} > \frac{15t}k,$ completing the proof.
\end{proof}

Being more careful in the estimates in the above argument the choice of $k$ in Theorem \ref{thm:rodlfromerdos} can be improved to $k = 200\left\lceil \frac{1}{\eps}\log^2\frac{1}{\eps} f(\frac{1}{\eps}\log^4\frac{1}{\eps})^2\right\rceil$. In particular, if $H$ has the Erd\H{o}s-Hajnal property, then taking $f$ to be a constant function, with a worse constant factor, in Theorem \ref{thm:rodlfromerdos} in the choice of $k$ we can replace the exponent $4$ of the logarithm by $2$. \Cref{mainthmquant} follows by taking $f$ to be the constant function with value $1/c_H$. 

The same proof gives the following variant, replacing the application of 
Lemma \ref{lem:counts-lwr-bnd} by \Cref{lem:counts-lwr-bnd-close}.

\begin{thm}\label{thm:rodlfromerdoshered}
    Let $0<\eps< \frac{1}{100}$ and $k:=200\ceil{\frac1{\eps}\log^4 \frac1{\eps}}$. Let $1/\delta=16k^{f(k)-1}$. 
    Let $\mathcal{F}$ be a hereditary family of graphs which has the $f$-Erd\H{o}s-Hajnal property. Let $G$ be a graph on $n$ vertices which is $(\delta/k)$-close to $\mathcal{F}$. Then there is a vertex subset $S \subset V(G)$ with $|S| \geq \delta n$ such that the induced subgraph $G[S]$ or its complement $\bar G[S]$ has maximum degree at most $\eps (|S|-1)$.
\end{thm}

\Cref{mainfamily,mainfamily2} follow as immediate corollaries of \Cref{thm:rodlfromerdoshered}. \Cref{corperfect} also follows since for the family of perfect graphs we have $f(k)=2$ and as mentioned above we may take $k = 200\left\lceil \frac{1}{\eps}\log^2\frac{1}{\eps} f(\frac{1}{\eps}\log^4\frac{1}{\eps})^2\right\rceil$.

\section{Extensions: hypergraphs, ordered graphs, tournaments, and colorings}\label{extensions}

In this section, we discuss equivalences of the Erd\H{o}s-Hajnal property, the polynomial R\"odl property and the viral property for hypergraphs, tournaments, ordered graphs, and edge-colored complete graphs. We will begin with the results on hypergraphs.  

The hypergraph container method \cite{BMS,ST} generalizes the Kleitman-Winston graph container method \cite{KW80,KW82} to $r$-uniform hypergraphs for $r\ge 3$. It shows that under appropriate degree conditions on $\ell$-tuples of vertices for $\ell \le r-1$, the independent sets in an $r$-uniform hypergraph on $n$ vertices are contained in a relatively small number of containers, each of them a subset of vertices of size much smaller than $n$. 

The general hypergraph container lemma \cite{BMS, ST} was proved using the scythe algorithm. We next state and derive a hypergraph analog of Lemma \ref{lem:counts-upr-bnd}, which is a version of the hypergraph container lemma in the dense case and which has a much simpler proof. We note that a version of this lemma with worse constant factors can be derived from the key container lemma in \cite{BMS} and that an argument of a similar flavor appears in \cite{Kmm}. We include below a different simple proof that works in our setting. This lemma will be key to the derivation of analogs of Theorem \ref{thm:rodlfromerdos} for hypergraphs and tournaments.

\begin{lem}\label{lem:hyp-counts-upr-bnd}
Let $r\ge 2$. Let $\ell \geq 0$ and $k \geq  (r-1)\ell$ be integers, $u,n$ be positive integers, and $0< \varepsilon \leq 1$ be such that $(1-\varepsilon)^{\ell}n \leq u$. Let $G$ be an $r$-uniform hypergraph on $n$ vertices in which any vertex subset $S$ with $|S| > u$ is such that the induced subgraph $G[S]$ has maximum degree at least $\varepsilon (|S|-1)^{r-1}$. Then the number of independent sets of order $k$ in $G$ is at most $$\binom{n}{(r-1)\ell}\binom{u}{k-(r-1)\ell}.$$
\end{lem}

\begin{proof}
For each subset $U$ of vertices of $G$ and a tuple $J$ consisting of $|J|\le r-2$ vertices, we consider $G_J$ to be the $(r-|J|)$-uniform hypergraph where $e'$ is an edge of $G_J$ if $e'\cup J\in E(G)$. We also consider an ordering $\pi_{U,J}$ of vertices $v$ of $U\setminus J$ defined recursively as follows. The first vertex in the ordering $\pi_{U,J}$ is a vertex $v \in U\setminus J$ maximizing $\deg_U(J\cup \{v\})$ (breaking ties by an underlying ordering of the vertices), where $\deg_U(J\cup \{v\})$ is the number of edges of $G[U]$ containing $J\cup \{v\}$. The ordering on $U\setminus \{v\}$ is then given by $\pi_{U\setminus \{v\}, J}$. 

For each independent set $I$ in $G$, we let $j_1$ be the first vertex of $I$ in the ordering $\pi_{V(G),\emptyset}$. Let $D_1$ denote the set of vertices preceding $j_1$ in $\pi_{V(G),\emptyset}$ and $A_1 = D_1 \cup \{j_1\}$. For $t\in [2,r-1]$, we then let $j_t$ be the first vertex of $I$ in the ordering $\pi_{V(G) \setminus A_{t-1},\{j_1,\dots,j_{t-1}\}}$. Let $D_t$ be the set of vertices preceding $j_t$ in $\pi_{V(G) \setminus A_{t-1},\{j_1,\dots,j_{t-1}\}}$ and let $A_t = A_{t-1}\cup D_t \cup \{j_t\}$. If $|I| \geq r-1$, we define the initial segment of $I$ as $J^{(1)}=(j_1,\dots,j_{r-1})$. 

Let $D'$ denote the set of vertices adjacent to $j_{r-1}$ in $G_{\{j_1,\dots,j_{r-2}\}}$. Observe that $I\setminus J^{(1)}$ is an independent set in $G[V(G)\setminus (D' \cup A_{r-1})]$. We next verify our key claim that 
\begin{equation}\label{eq:drop}
|V(G)\setminus (D' \cup A_{r-1})| \le n-(r-1)-\varepsilon(n-1),
\end{equation}
as long as $n - |A_1| \ge u$. 

By our assumption, the vertex $j_1$ is adjacent to at least $\varepsilon (n-|A_1|)^{r-1}$ edges in $G[V(G)\setminus A_1]$.  
We claim that the number of edges of $G_{\{j_1,\dots,j_{t-1}\}}[V(G)\setminus A_{t-1}]$ is at least $(n-|A_1|)^{r-t+1}(\varepsilon - \sum_{t'\le t-1} |D_{t'}|/(n-|A_1|))$. We prove this claim by induction on $t$. Indeed, this holds for the base case $t=2$. Assuming the claim holds for $t$, we have that the number of edges of $G_{\{j_1,\dots,j_{t}\}}[V(G)\setminus A_{t}]$ is at least 
\[
\frac{1}{n-|A_1|} \left(|E(G_{\{j_1,\dots,j_{t-1}\}}[V(G)\setminus A_{t-1}])| - |D_{t}|(n-|A_1|)^{r-t}\right) \ge (n-|A_1|)^{r-t} \left(\varepsilon - \sum_{t'\le t} |D_{t'}|/(n-|A_1|)\right),
\]
completing the proof by induction. By the claim above, we have that 
\[
|D'| \ge (n-|A_1|) \left(\varepsilon - \sum_{t'\le r-1} |D_{t'}|/(n-|A_1|)\right). 
\]
Hence, 
\[
|D'| + \sum_{t'\le r-1} |D_{t'}| \ge \varepsilon (n-|A_1|). 
\]
As such, we can specify each independent set $I$ of order $k$ by its initial segment $J^{(1)}$, and an independent set of size $k-(r-1)$ in $G[V(G)\setminus (D' \cup A_{r-1})]$, which has at most $n - |D'| - \sum_{t'\le r-1}(|D_{t'}|+1) \le n - \varepsilon (n-1) - (r-1) \leq (1-\varepsilon)n$ vertices. This shows (\ref{eq:drop}) as desired. 

After identifying the initial segment $J^{(1)}$, we replace $G$ by $G[V(G)\setminus (D' \cup A_{r-1})]$ and $I$ by $I\setminus J^{(1)}$. We then repeat the same procedure. Denote the initial segments obtained by $J^{(1)}, J^{(2)},\dots, J^{(\ell)}$. Denote the vertex set remaining in consideration after identifying $J^{(s)}$ by $V^{(s)}$ for $s\le \ell$. By (\ref{eq:drop}), we have that $|V^{(s+1)}|\le (1-\varepsilon)|V^{(s)}|$ whenever $|V^{(s)}| > u$. By our assumption on $\ell$ we then have $|V^{(\ell)}| \le u$. 
We then have a vertex subset $|U|\le u$ that is determined by $J^{(1)},\dots,J^{(\ell)}$, for which $I\setminus \bigcup_{s\le\ell}J^{(s)} \subseteq U$. 

 Furthermore, observe that the ordered sequence of vertices in $J^{(1)},J^{(2)},\dots,J^{(s)}$ can be determined from the (unordered) set of vertices $\bigcup_{s'\le s} J^{(s')}$. Indeed, given this set of vertices, we can recover $j_1\in J^{(1)}$ as the vertex in the set first according to $\pi_{V(G),\emptyset}$, from which we can sequentially determine the remaining vertices of the initial segment $J^{(1)}$. The same procedure then allows us to sequentially determine the other segments $J^{(2)},\dots,J^{(s)}$. 

Therefore, the number of independent sets $I$ of $G$ of order $k$ is at most 
\[
\binom{n}{(r-1)\ell} \binom{u}{k-(r-1)\ell},
\]
where the first term corresponds to the number of choices for $J^{(1)},\dots,J^{(\ell)}$, and the second term corresponds to the number of extensions of $\bigcup_{s\le\ell}J^{(s)}$ to an independent set $I$ of order $k$. 
\end{proof}

From Lemma \ref{lem:hyp-counts-upr-bnd}, using identical arguments as in the proof of Theorem \ref{thm:rodlfromerdos}, we can derive the hypergraph analog of Theorem \ref{thm:rodlfromerdos}. Given a non-decreasing function $f:\mathbb{R}^+ \to \mathbb{R}^+$ such that $2 \le f(k) =O(\log k)$, as in the graph case, we say that a hereditary family $\mathcal{F}$ of $r$-uniform hypergraphs has the $f$-Erd\H{o}s-Hajnal property if for any $k \geq 1$ any $r$-uniform hypergraph $G \in \mathcal{F}$ with at least $k^{f(k)}$ vertices has a homogeneous set of size at least $k$. An $r$-uniform hypergraph $G$ on $n$ vertices is \emph{$\alpha$-close} to a family $\mathcal{F}$ of $r$-uniform hypergraphs if one can add or delete at most $\alpha n^2/r!$ edges from $G$ to obtain a hypergraph in $\mathcal{F}$.

\begin{thm}\label{thm:hyp-rodlfromerdos} Let $\mathcal{F}$ be a hereditary family of $r$-uniform hypergraphs which has the $f$-Erd\H{o}s-Hajnal property. 
    There exists $C>0$ such that the following holds. Let $0<\eps< \frac{1}{100}$ and $k:=C\ceil{\frac1{\eps}\log^4 \frac1{\eps}}$. Let $1/\delta=16k^{f(k)-1}$.  If an $r$-uniform hypergraph $G$ on $n$ vertices is $(\delta/k)$-close to $\mathcal{F}$, then $G$ has an induced subgraph $G[S]$ with $|S| \ge \delta n$ with maximum degree at most $\eps (|S|-1)^{r-1}$ or minimum degree at least $(1-\eps)(|S|-1)^{r-1}$. 
\end{thm}

The proof follows that of Theorem \ref{thm:rodlfromerdos}. Indeed, following the proof of Lemma \ref{lem:counts-lwr-bnd-close}, the $r$-uniform hypergraph $G$ contains at least $\frac{1}{2}(n/(2t))^k$ homogeneous sets of order $k$ with $t=k/\delta$. 
Repeating the proof of Theorem \ref{thm:rodlfromerdos}, for an appropriate $C>0$, choosing $k = C\frac{1}{\eps} \log^{4} \frac{1}{\eps}$, $1/\delta = Ck^{f(k)-1}$ and $\ell = \lceil\frac{1}{\eps}\log \frac{1}{\delta}\rceil$, we obtain Theorem \ref{thm:hyp-rodlfromerdos}. 

Several previous works have studied analogs of the Erd\H{o}s-Hajnal conjecture for hypergraphs \cite{CFS12, AST, GT22,E-H-hypergraph-families,rodl-schacht}. In the case the hereditary family $\mathcal{F}$ is the family of $r$-uniform hypergraphs avoiding a specific induced $r$-uniform hypergraph $H$ with $r \geq 3$, Gishboliner and Tomon \cite{GT22} show that $\mathcal{F}$ satisfies the Erd\H{o}s-Hajnal property (that is, the $f$-Erd\H{o}s-Hajnal property for $f(k) = C_H$) only when $|V(H) \leq r$ (the trivial cases) or when $r=3$ and $H$ is the $3$-uniform hypergraph with two edges on four vertices. Theorem \ref{thm:hyp-rodlfromerdos} implies that the polynomial R\"odl and viral strengthenings hold in these cases. There are more instances of families $\mathcal{F}$ satisfying the Erd\H{o}s-Hajnal property when $\mathcal{F}$ is defined by forbidding more than just a single hypergraph $H$, see \cite{E-H-hypergraph-families} for some examples. Our results give the polynomial R\"odl and viral strengthenings of these results as well.

\noindent {\bf Ordered graphs.}
An ordered graph is a graph with a linear ordering on its vertex set. Given an ordered graph $H$, an induced copy of $H$ in an ordered graph $G$ is given by an order preserving map from $V(H)$ to $V(G)$ which maps edges of $H$ to edges of $G$ and non-edges of $H$ to non-edges of $G$. We can define the Erd\H{o}s-Hajnal, polynomial R\"odl and viral properties for ordered graphs as in the unordered case. These notions have been studied before in \cite{APS01,buildable,long-paths,tung-paper}. Our proofs of the equivalences for the unordered case carry through identically to give the analogous results for the ordered case. 

\noindent {\bf Tournaments.}
As another application of Lemma \ref{lem:hyp-counts-upr-bnd}, we derive the analog of \Cref{thm:rodlfromerdos} for tournaments. Say that a tournament $T$ has the \emph{Erd\H{o}s-Hajnal property} if there is a constant $c_T>0$ such that every $n$-vertex tournament that does not contain a copy of $T$ contains a transitive subtournament on $n^{c_T}$ vertices. Alon, Pach and Solymosi \cite{APS01} proved that the Erd\H{o}s-Hajnal conjecture is equivalent to the statement that every tournament has the Erd\H{o}s-Hajnal property. Since then, the tournament Erd\H{o}s-Hajnal property has also received considerable attention, see for example \cite{BCC15,BCC19,BCCFLSST,BCCZ22,Choro18,maria-survey,CSSS24,HLTW19,buildable,tung-paper,VC-dimension,ZG23,ZG23b,ZG24,ordered-paths}.

As in the graph case, for a non-decreasing function $f:\mathbb{R}^+ \to \mathbb{R}^+$ such that $2 \le f(k) \le \log k$ for all $k \geq 1$, we say that a tournament $T$ has the $f$-Erd\H{o}s-Hajnal property if for any $k \geq 1$ and any $T$-free tournament $G$ with at least $k^{f(k)}$ vertices, $G$ has a transitive subtournament of size at least $k$.

We next discuss the corresponding polynomial R\"odl and viral properties for tournaments. We say that a subset $S$ of vertices of a tournament is \emph{$\eps$-transitive} if the induced subtournament on $S$ can be made transitive by changing the direction of at most $\eps \binom{|S|}{2}$ edges. The polynomial R\"odl property for $T$ says that there is a constant $C_T$ such that every $T$-free tournament on $n$ vertices contains an $\eps$-transitive subset of vertices of size at least $\eps^{C_T}n$. Tournament $T$ is said to be viral if there are constants $C_T,D_T$ such that any $n$-vertex tournament with at most $\eps^{D_T} n^{|T|}$ copies of $T$ contains an $\eps$-transitive subset of vertices of size at least $\eps^{C_T}n$. 

The following theorem gives the equivalence for tournaments between the viral, the polynomial R\"odl, and the Erd\H{o}s-Hajnal property.

\begin{thm}\label{thm:tour-rodlfromerdos}
    Let $C,C'>0$ be appropriate constants. Let $0<\eps< \frac{1}{100}$ and $k:=C\ceil{\frac1{\eps^2}\log^4 \frac1{\eps}}$. Let $1/\delta=C'k^{f(k)-1}$. 
    Let $T$ be an $h$-vertex tournament which has the $f$-Erd\H{o}s-Hajnal property. Let $G$ be an $n$-vertex tournament with less than $(\delta /k)^{h-1} n^{h}$ induced copies of $T$. Then $G$ has an induced subtournament $G[S]$ with $|S| \ge \delta n$ which is $\eps$-transitive. 
\end{thm}
The proof of Theorem \ref{thm:tour-rodlfromerdos} is mostly identical to the proof of Theorem \ref{thm:rodlfromerdos}, apart from the new input which is Lemma \ref{lem:hyp-counts-upr-bnd}, and the following result of Fox and Sudakov \cite[Lemma 1.3]{FS08}. 
\begin{lem}\label{lem:tour-to-hyp}
    There exists $c>0$ such that the following holds. If a tournament on $m$ vertices contains less than $c\eps^2 m^3$ directed triangles, then it is $\eps$-transitive. 
\end{lem}


\noindent \emph{Proof sketch of \Cref{thm:tour-rodlfromerdos}.} Given a tournament $G$ on $n$ vertices, we construct the following auxiliary $3$-uniform hypergraph $G^{(3)}$ on $V(G)$, where a triple of vertices is an edge of $G^{(3)}$ if and only if it forms a directed triangle (i.e.\ is not transitive) in $G$. We may assume $G$ has no $\eps$-transitive subset of size at least $\delta n$ as otherwise we are done. By \Cref{lem:tour-to-hyp}, every subset $S$ with $|S| \geq \delta n$ has at least $\theta |S|^3$ directed triangles, where $\theta:=c \eps^2$. This implies by averaging that for any $S \subset V(G)$ with $|S| \geq \delta n$ the maximum degree of $G^{(3)}[S]$ is at least $\theta (|S|-1)^2$. 
Let $\ell:= \lceil \theta^{-1} \log(1/\delta)\rceil$ so that $(1-\theta)^\ell \le \delta$. Applying Lemma \ref{lem:hyp-counts-upr-bnd}, the number of independent sets in $G^{(3)}$ of size $k$ is at most 
\[
\binom{n}{2\ell} \binom{\delta n}{k-2\ell}.
\]
Note that independent sets in $G^{(3)}$ are precisely the transitive subtournaments of $G$. On the other hand, an argument identical to the proof of Lemma \ref{lem:counts-lwr-bnd} implies that the number of transitive subtournaments in $G$ of size $k$ is at least $\frac{1}{2} (n/t)^k$ where $t=k^{f(k)}$. 

Now, an argument identical to the proof of Theorem \ref{thm:rodlfromerdos} yields that with our choice of parameters, 
\[
\binom{n}{2\ell} \binom{\delta n}{k-2\ell} > \frac{1}{2} (n/t)^k, 
\]
a contradiction. This completes the proof sketch of Theorem \ref{thm:tour-rodlfromerdos}.\qed 

Note that, for tournaments, unlike with graphs, we lost a factor $2$ in the constant in the exponent going back and forth between the Erd\H{o}s-Hajnal property and the polynomial R\"odl property. This comes from the application of Lemma \ref{lem:tour-to-hyp} and considering the auxiliary hypergraph. 

Another proof of these equivalences for tournaments that gives similar quantitative bounds can be achieved by going through ordered graphs instead of hypergraphs. Indeed, for a tournament $T$, consider the family $\mathcal{H}_T$ of ordered graphs on $V(T)$, where an ordered graph $H$ on $V(T)$ is in $\mathcal{H}_T$ if there is a vertex ordering of $T$ such that the edges of $H$ are precisely the forward edges of $T$ with respect to the given vertex ordering. Note that a copy of $T$ in a tournament $G$ corresponds to a copy of some ordered graph in $\mathcal{H}_T$ in $G'$, an ordered graph on $V(G)$ obtained by arbitrarily ordering $V(G)$ and whose edges are the forward edges of $G$ with respect to this ordering. While homogeneous sets in the ordered graph $G'$ are transitive subtournaments of the corresponding tournament $G$, the converse is not true. However, every transitive subtournament on $(k-1)^2+1$ vertices in $G$ contains a homogeneous set of size $k$ in $G'$ through an application of the Erd\H{o}s-Szekeres theorem stating that every permutation of $m$ distinct real numbers contains an increasing or decreasing subsequence of length $\lceil m \rceil$. This is where the factor $2$ in the exponent occurs in this alternative proof as well. 

\noindent {\bf Edge-colorings.}
Let $\chi$ be an $r$-coloring of the edges of $K_h$. We say that $\chi$ has the {\it $r$-colored Erd\H{o}s-Hajnal property} if in every $r$-coloring of the edges of $K_n$ with no induced subgraph having an edge coloring isomorphic to $c$, there is a subset of vertices of size $n^{c_\chi}$ which misses one of the $r$ colors. The study of this property goes back to the original work of Erd\H{o}s and Hajnal \cite{erdos-hajnal-2}, and has also received considerable interest since then (see, for example, \cite{CFR,FGP,FL20,axenovich2023note,ASW22,Hajnal}).

Similarly, we say that $\chi$ has the $r$-colored polynomial R\"odl property if in every $r$-coloring $K_n$ with no induced subgraph having an edge coloring isomorphic to $c$, there is a subset of vertices of size $\eps^{C_\chi}n$ in which one of the $r$ colors has density at most $\eps$. Observe that for $r=2$, these properties are identical to the usual Erd\H{o}s-Hajnal and polynomial R\"odl properties in graphs. 

Our proof goes through to show the equivalence between the properties in the $r$-colored case. In particular, by the same argument in Lemma \ref{lem:counts-lwr-bnd} and the pigeonhole principle, assuming the $r$-colored Erd\H{o}s-Hajnal property, there is one color class for which the number of vertex subsets of size $k$ missing this color is at least $(1/r)(n/t)^k$ where $t=k^{1/c_\chi}$. On the other hand, by considering the graph $G$ given by this color class, and assuming that there is no induced subgraph of $G$ on $\delta n$ vertices with edge density at most $\eps$ for $\delta = \eps^{1/c_\chi - 1 + o(1)}$, Lemma \ref{lem:counts-upr-bnd} implies that the number of vertex subsets of size $k$ missing the color is at most $\binom{n}{\ell} \binom{\delta n}{k-\ell}$ for $\ell = \lceil\frac{1}{\eps}\log \frac{1}{\delta}\rceil$. The same calculation as in the proof of Theorem \ref{thm:rodlfromerdos} yields a contradiction, and thus the $r$-colored polynomial R\"odl holds for the coloring $\chi$. 

\section{Concluding Remarks}

As already showcased, our argument is quite flexible. For example, one can use it to prove viral type results for hereditary families as well, as we discuss next. Let $\mathcal{F}$ be a hereditary family of graphs and let $\mathcal{H}=\mathcal{H}_{\mathcal{F}}$ denote the family of minimal graphs not in $\mathcal{F}$. For example, if $\mathcal{F}$ is the family of perfect graphs, then the strong perfect graph theorem \cite{CRST} says that $\mathcal{H}$ is the family of odd cycles of length at least five and their complements. A celebrated result of Alon and Shapira \cite{AS08} in property testing (see also the survey \cite{CF2013}), which is an infinite induced graph removal lemma, shows that a graph is $o(1)$-close to a hereditary family $\mathcal{F}$ if and only if each bounded-sized graph in $\mathcal{H}$ has induced subgraph density $o(1)$. However, the quantitative dependence can be arbitrarily bad \cite{GS2020}.

For each $H \in \mathcal{H}$, let $\eps_H \geq 0$ and $\sum_{H \in \mathcal{H}} \eps_H \leq 1$. 
We next discuss precisely a variant of the viral property for hereditary families. A family $\mathcal{F}$ of graphs has the {\it $f$-Erd\H{o}s-Hajnal property} if every graph $G \in \mathcal{F}$ with at least $k^{f(k)}$ vertices satisfies $\hom(G) \geq k$. The conclusion of Theorem \ref{thm:rodlfromerdos} goes through for hereditary families $\mathcal{F}$ with the condition that $\mathcal{F}$ has the $f$-Erd\H{o}s-Hajnal property and for each $H \in \mathcal{H}$ the number of induced copies of $H$ in $G$ is at most $\eps_{H}(\delta/k)^{h-1}n^k$. Indeed, the main modification is in the statement of \Cref{lem:counts-lwr-bnd}, where the assumption would be that the number of induced copies of $H$ in $G$ is at most $\eps_{H}(\delta/k)^{h-1}n^k$ for each $H \in \mathcal{H}$. The proof of this generalization of \Cref{lem:counts-lwr-bnd} goes through in the same way by sampling and deleting one vertex from each induced copy of a graph in $\mathcal{H}$ in the sample. 

It follows from the strong perfect graph theorem \cite{CRST} by taking $\eps_H=|H|^{-2}$ for each $H \in \mathcal{H}$ that if a graph $G$ is such that $G$ and its complement $\bar G$ have for each odd integer $h \geq 5$ at most $\eps^{2(h-1)+o(1)}n^h$ induced cycles of length $h$, then $G$ has an $\eps$-homogeneous set of size $\eps^{1+o(1)}|G|$. 

For another example, Nguyen, Scott and Seymour \cite{VC-dimension} extended the notion of viral graphs to hereditary families. The family of graphs whose neighborhood set system has VC-dimension at most $d$ has a finite number of minimal forbidden induced subgraphs, so we can take $\eps_H$ to be the constant $1/|\mathcal{H}|$ in this case. In such a case, being viral is equivalent to having the Erd\H{o}s-Hajnal property. A special case of this is a recent result of Gishboliner and Shapira \cite{lior-P4}. Note that our proof avoids the use of the strong regularity lemma for graphs with bounded VC-dimension of Lov\'asz and Szegedy \cite{lovasz-szegedy} (see also \cite{regularity1, regularity2} for variants). Since the Erd\H{o}s-Hajnal property has recently been established for the family of graphs of VC-dimension at most some $d$ in \cite{VC-dimension} (in fact they conclude, via an application of the aforementioned regularity lemma that their argument implies the polynomial R\"odl property) this implies the viral version also holds. See \cite{VC-dimension} for more details and precise definitions.

We believe it would be interesting to determine if the polylogarithmic factors can be removed in our results.

\textbf{Note added in proof.} After our paper appeared on arXiv, Tung Nguyen notified us that he has a direct proof of \Cref{cor:string} obtained by modifying the argument of Tomon \cite{tomon2023string}. We would also like to thank J\'ozsef Balogh, Robert Morris and Wojciech Samotij for helpful discussions on hypergraph containers. 

\providecommand{\MR}[1]{}
\providecommand{\MRhref}[2]{%
  \href{http://www.ams.org/mathscinet-getitem?mr=#1}{#2}
}


\providecommand{\bysame}{\leavevmode\hbox to3em{\hrulefill}\thinspace}
\providecommand{\MR}{\relax\ifhmode\unskip\space\fi MR }
\providecommand{\MRhref}[2]{%
  \href{http://www.ams.org/mathscinet-getitem?mr=#1}{#2}
}
\providecommand{\href}[2]{#2}

\end{document}